\documentclass[twoside]{scrartcl}
\ifdefined\Kindle
\usepackage[margin=1mm,papersize={110mm,75mm}]{geometry}
\AtBeginDocument{\presetkeys{todonotes}{inline}}
\fi

\usepackage{ifluatex}

\usepackage{amsmath,amsfonts,amsthm,braket}
\usepackage{enumitem}
\usepackage[unicode]{hyperref}

\ifluatex
\usepackage{polyglossia}
\setmainlanguage{english}
\usepackage{unicode-math}

\AtBeginDocument{\directlua{require("combining.lua")}}
\newcommand{{\dot \}}{\relax\ifmmode\dot\else\.\fi}
\def{\hat \} {\hat}
\def\́ {\acute}
\def{\overline \} {\overline}
\def{\tilde \} {\tilde}
\def\͠{\widetilde}
\else
\usepackage{amssymb,MnSymbol}
\usepackage[utf8]{inputenc}
\usepackage[english]{babel}
\usepackage{expl3}
\newcommand{\eps}{\varepsilon}
\fi

\usepackage{csquotes}
\usepackage[backend=biber,firstinits=true,sortcites=true,url=false]{biblatex}
\addbibresource{complex-foliations.bib}
\addbibresource{bifurcations.bib}

\usepackage{datetime}
\usepackage[automark]{scrlayer-scrpage}
\pagestyle{scrheadings}
\ohead{\headmark}
\chead{}
\ihead{Genera of non-algebraic leaves}
\ofoot{\pagemark}

\usepackage[final]{ifdraft}
\ifdraft{
\newcommand{\draftmark}{\small {\upshape\bfseries Draft version} from \csname @date\endcsname}
\ifoot[\draftmark]{\draftmark}
\PassOptionsToPackage{draft}{todonotes}
\AtBeginDocument{\date{\today\ \currenttime}}
}{}
\usepackage[backgroundcolor=white,textsize=scriptsize,obeyDraft]{todonotes}

\AtBeginDocument{

}

\DeclareMathOperator{\const}{const}
\DeclareMathOperator{\Aff}{Aff}
\DeclareMathOperator{\codim}{codim}
\DeclareMathOperator{\tr}{tr}
\DeclareMathOperator{\Isohol}{Isohol}
\DeclareMathOperator{\Topo}{Topo}

\newcommand{\wiki}[2]{\href{https://en.wikipedia.org/wiki/#1}{#2}}
\newcommand{\orcid}[1]{ORCID ID: \href{http://orcid.org/#1}{#1}}
\newcommand{\MRauthor}[1]{MR Author ID: \href{http://www.ams.org/mathscinet/search/author.html?mrauthid=#1}{#1}}

\ExplSyntaxOn
\newcommand{\mynewtheorem}[1]{
\newtheorem{#1}{\tl_upper_case:n #1}
\newtheorem*{#1*}{\tl_upper_case:n #1}
\expandafter\edef\csname #1autorefname\endcsname{\tl_upper_case:n #1}
}
\ExplSyntaxOff

\mynewtheorem{theorem}
\mynewtheorem{lemma}
\mynewtheorem{conjecture}
\mynewtheorem{corollary}

\theoremstyle{remark}
\mynewtheorem{remark}

\theoremstyle{definition}
\mynewtheorem{definition}

\title{Genera of non-algebraic leaves of~polynomial foliations of~\texorpdfstring{${\mathbb C}^2$}{ℂ²}}
\author{
	Nataliya Goncharuk%
    \thanks{Cornell University, College of Arts and Sciences, Department of Mathematics, 310 Mallot Hall, Ithaca, NY, 14853, US}
    \thanks{Higher School of Economics, Department of Mathematics, 20 Myasnitskaya street, Moscow 101000, Russia}
    \thanks{\orcid{0000-0002-4270-0510}, \MRauthor{978548}}
	\thanks{The research was supported by Russian President grant No. MD-2859.2014.1.}
	\thanks{Research was supported by RFBR project 13-01-00969-a.}
	\thanks{Research was supported by Proyecto IN103914 PAPIIT (DGAPA) Universidad Nacional Autónoma de México.}
	\and
	Yury Kudryashov%
    \footnotemark[1]
    \footnotemark[2]
    \footnotemark[5]\,
    \footnotemark[6]
    \thanks{\orcid{0000-0003-4286-9276}, \MRauthor{914251}}\,\,
	\thanks{Research carried out within “The National Research University Higher School of Economics’ Academic Fund Program” in 2014--2015, research grant No. 14-01-0193.}
}
\makeatletter
\hypersetup{pdftitle=\@title,pdfauthor={Nataliya Goncharuk, Yury Kudryashov}}
\makeatother
\begin{document}
\maketitle
\todo[inline,author=YK]{
Acknowledge grants supported us when we were rewriting the paper in 2017.
}
\begin{abstract}
	In this article, we prove two results.
	First, we construct a dense subset in the space of polynomial foliations of degree $n$ such that each foliation from this subset has a~leaf with at least $\frac{(n+1)(n+2)}2-4$ handles.
	Next, we prove that for a generic foliation invariant under the map $(x, y)\mapsto (x, -y)$ all leaves (except for a finite set of algebraic leaves) have infinitely many handles.
\end{abstract}
\section{Introduction}

Consider a~polynomial differential equation in ${\mathbb C}^2$ (with complex time),

\begin{equation}
	\label{polynomial.vf}
	\begin{aligned}
		{\dot x} &= P(x,y),\\
		{\dot y} &= Q(x,y),
	\end{aligned}
\end{equation}
where $\max (\deg P, \deg Q) =n$.
The splitting of ${\mathbb C}^2$ into trajectories of this vector field defines a singular analytic foliation of ${\mathbb C}^2$.

Denote by ${\mathcal A}_n$ the space of foliations of ${\mathbb C}^2$ defined by vector fields \eqref{polynomial.vf} of degree at most $n$ with coprime $P$ and $Q$.
Denote by ${\mathcal B}_n$ the space of foliations of ${\mathbb C}P^2$ defined by a polynomial vector field \eqref{polynomial.vf} of degree at most $n$ in each affine chart.
It is easy to show that ${\mathcal A}_n\subset {\mathcal B}_{n+1}\subset {\mathcal A}_{n+1}$.

More geometrically, ${\mathcal B}_n$ can be defined as the space of foliations of ${\mathbb C}P^2$ of projective degree $n$, and ${\mathcal A}_n\setminus {\mathcal B}_n\subset {\mathcal B}_{n+1}$ as the set of foliations of projective degree $n+1$ tangent to the line at infinity, see e.g. \cite[Sec. 25A]{IYbook} for details.

Numerous studies in this field are devoted to the properties of generic foliations from ${\mathcal A}_n$ and ${\mathcal B}_n$, see \cite{Shch06:en} for a~survey.
In particular, the leaves of a generic foliation ${\mathcal F}\in {\mathcal A}_n$, $n\geq 2$, are dense in ${\mathbb C}^2$, see \cite{KhV62:en,Il78:en,Shch84:trans};
see also \cite{M75:en,LR03} for similar results about \emph{locally} generic foliations ${\mathcal F}\in {\mathcal B}_n$, $n\geq 3$.
Another classical question concerns degree and genus of an algebraic leaf of a polynomial foliation, see \cite{LN02}.
We study genera of non-algebraic leaves.

For a generic \emph{analytic} foliation, the question about the topology of a leaf was studied by T.~Firsova and T.~Golenishcheva--Kutuzova.

\begin{theorem*}
	[see \cite{F06:en,K06:en}]
	Among leaves of a generic analytic foliation, countably many are topological cylinders, and the rest are topological discs.
\end{theorem*}

For a generic polynomial foliation, the analogous result is not known.
The fact that almost all leaves are topological discs would follow from Anosov conjecture on identical cycles.

\begin{definition}
	\label{def:identical-cycle}
	\label{def:complex-LC}
    Let $L$ be a leaf of ${\mathcal F}$.
    A non-trivial free homotopy class $[\gamma ]$, $\gamma :S^1\to L$, is called an \emph{identical cycle} if the holonomy along any of its representatives $\gamma $ is identical, and is called a \emph{complex limit cycle} otherwise.
\end{definition}
The holonomy maps along different representatives are conjugate, so it is not important which representative we choose.

\begin{conjecture}
	[D. Anosov]
	A generic polynomial foliation has no identical cycles.
\end{conjecture}

In \autoref{sec:many-handles}, we give a partial answer to the question:
“What topological structures of the leaves can arise in a \emph{dense} subset of ${\mathcal A}_n$?”.
Namely, we prove the following theorem.

\begin{theorem}
	\label{positive.genus}
	For each $n\geq 2$, the set of polynomial foliations having a~leaf with at least $\frac{(n+1)(n+2)}2-4$ handles is dense in ${\mathcal A}_n$.
\end{theorem}

The statement of this theorem is inspired by the following theorem due to D.~Volk \cite{V06:en}.

\begin{theorem}
	[Density of foliations with separatrix connections]
	\label{volk.main}
	For each $n\geq 2$, the set of polynomial foliations having a~separatrix connection is dense in ${\mathcal A}_n$.
\end{theorem}

We also prove a refined version of Volk's theorem, see \autoref{sub:volk}, and use it to prove \autoref{positive.genus}.

In \autoref{sec:inf-genus}, we obtain the following result:

\begin{theorem}
	\label{infinite.genus}
	Let ${\mathcal A}_n^{sym}$ (resp., ${\mathcal B}_n^{sym}$) be the subspace of ${\mathcal A}_n$ (resp., ${\mathcal B}_n$) given by

	\begin{align}
		\label{vf.symmetry}
		P(x, -y)&=-P(x, y),& Q(x, -y)&=Q(x,y).
	\end{align}

	Take $n\geq 2$.
	For any foliation ${\mathcal F}$ from some open dense subset of ${\mathcal A}_n^{sym}$ (resp., ${\mathcal B}_{n+1}^{sym}$), all leaves of ${\mathcal F}$ (except for a finite set of algebraic leaves) have infinite genus.
\end{theorem}

There are some unpublished earlier results in this direction.
For generic homogeneous vector fields, almost all leaves have infinite genus;
the proof is due to Yu.\ Ilyashenko, but it was never written down.
We write it in \autoref{sec:proof-Il}.

In the unpublished version of his thesis, V.\ Moldavskis \cite{MolTh} proves that for a generic vector field of degree $n\geq 5$ with \emph{real coefficients} and the symmetry \eqref{vf.symmetry} each leaf has infinitely-generated first homology group.
However this is only a draft text, so the proof lacks some details and has some gaps.

\section{Preliminaries}

In this section we shall recall some results and introduce required notions and notation.
In some cases we formulate refined versions of earlier results or provide explicit constructions.

\subsection{Genus of a leaf}
\label{sub:genus}

A leaf $L$ of a foliation is a (usually non-compact) Riemann surface.
Since it is not necessarily homeomorphic to a sphere with finitely many handles and holes, we shall provide a~definition of its genus.

\begin{definition}
	\label{def:genus-handles}
	A Riemann surface $L$ has \emph{at least} $g$ \emph{handles}, if it includes $g$ pairwise disjoint handles, i.\ e., subsets homeomorphic to the punctured torus.
	We say that $L$ has an \emph{infinite genus}, if it has at least $g$ handles for all natural $g$.
\end{definition}

Suppose that there exist $g$ pairs of closed loops $(c_1, c_2)$, $(c_3, c_4)$, \dots , $(c_{2g-1}, c_{2g})$ on $L$, such that $c_{2j-1}$ and $c_{2j}$ intersect transversally at exactly one point, and the loops from different pairs are disjoint.
Then a small neighborhood of each pair of cycles $(c_{2j-1}, c_{2j})$ is homeomorphic to a punctured torus, hence $L$ has at least $g$ handles.
We shall say that each pair $(c_{2j-1},c_{2j})$ \emph{generates} a handle.

\subsection{Extension to infinity}
Let us extend a polynomial foliation ${\mathcal F}\in {\mathcal A}_n$ given by \eqref{polynomial.vf} to ${\mathbb C}P^2$.
To this end, make the coordinate change $u=\frac 1 x, v = \frac y x$, and the time change $d\tau  = -u^{n-1} dt$.
The vector field takes the form

\begin{equation}
	\label{vf-uv}
	\begin{aligned}
		{\dot u} &= u{\tilde P}(u,v)\\
		{\dot v} &= v{\tilde P}(u,v) - {\tilde Q}(u,v)
	\end{aligned}
\end{equation}
where ${\tilde P}(u,v)= P\left(\frac 1 u, \frac v u\right) u^n$ and ${\tilde Q}(u,v)= Q\left(\frac 1 u,\frac v u\right) u^n$ are two polynomials of degree at most $n$.

Since ${\dot u}(0, v)\equiv 0$, the line at infinity $\set{u=0}$ is invariant under this vector field.
Denote by $h(v)$ the polynomial ${\dot v}(0, v)=v{\tilde P}(0,v) - {\tilde Q}(0,v)$.
There are two cases.

\begin{description}
	\item[Dicritical case, $h(v)\equiv 0$.]
		In this case \eqref{vf-uv} vanishes identically on $\set{u=0}$.
		Thus it is natural to consider the time change $d\tau =-u^n dt$ instead of $d\tau =-u^{n-1} dt$, and study the vector field
		\begin{align*}
			{\dot u} &= {\tilde P}(u,v)\\
			{\dot v} &= \frac{v{\tilde P}(u,v) - {\tilde Q}(u,v)}u
		\end{align*}
		whose trajectories are almost everywhere transverse to the line at infinity.
		This case corresponds to ${\mathcal B}_n\subset {\mathcal A}_n$.
	\item[Non-dicritical case, $h(v)\nequiv 0$.]
		In this case \eqref{vf-uv} has isolated singular points $a_j\in \set{u=0}$ at the roots of $h$, and $L_\infty = \set{u=0}\setminus \{a_1,a_2,\dots  \}$ is a leaf of the extension of ${\mathcal F}$ to ${\mathbb C}P^2$.
		Making the coordinate change $\left(\frac xy, \frac 1y\right)$, one can easily check that the only point of ${\mathbb C}P^2$ not covered by the affine charts $(x, y)$ and $(u, v)$ is a singular point of the extension of ${\mathcal F}$ to ${\mathbb C}P^2$ if and only if $\deg h<n+1$.
\end{description}

Denote by ${\mathcal A}_n'$ the set of foliations ${\mathcal F}\in {\mathcal A}_n$ such that $h$ has degree $n+1$, and $n+1$ distinct roots $a_j$, $j=1,\dots , n+1$.
In particular, all these foliations are non-dicritical.

Two vector fields define the same foliation ${\mathcal F}\in {\mathcal A}_n'$, if they are proportional, hence ${\mathcal A}_n'$ is a Zariski open subset of the projective space of dimension\footnote{From now on, “dimension” means “complex dimension”.} $(n+1)(n+2)-1$.
${\mathcal A}_n'$ is equipped with a natural topology induced from this projective space.

For each $j$, let $\lambda _j$ be the ratio of the eigenvalues of the linearization of \eqref{vf-uv} at $a_j$ (the eigenvalue corresponding to $L_\infty $ is in the denominator).
Note that $\lambda _j=\frac{{\tilde P}(0, a_j)}{h'(a_j)}$, hence
\begin{equation}
	\label{eqn:Linf-apart}
	\frac{{\tilde P}(0, v)}{h(v)}=\sum _{j=1}^{n+1}\frac{\lambda _j}{v-a_j}.
\end{equation}
Since ${\tilde P}(0, v)$ and $h(v)$ have the same leading coefficient, $\sum \lambda _j=1$.

\begin{remark}
	\label{rem:submersion}
	Any $(n+1)$-tuple of pairs $(a_j, \lambda _j)$ with pairwise distinct $a_j$ and $\sum \lambda _j=1$ corresponds to a foliation ${\mathcal F}\in {\mathcal A}_n'$.
	Indeed, it is enough to take ${\tilde P}(0, v)$ and $h(v)$ to be the numerator and the denominator of the right hand side in \eqref{eqn:Linf-apart}, then put ${\tilde Q}(0, v)=v{\tilde P}(0, v)-h(v)$.
	This uniquely defines the homogeneous components of $P$ and $Q$ of degree $n$, and one can arbitrarily choose homogeneous components of lower degrees.
\end{remark}
\subsection{Monodromy group}
\label{sub:mon-group}

For ${\mathcal F}\in {\mathcal A}_n'$, fix a non-singular point $O\in L_\infty $ and a cross-section $S$ at $O$ given by $v=\const$.
Let $\Omega L_\infty $ be the loop space of $(L_\infty , O)$, i.\ e., the space of all continuous maps $(S^1, pt)\to (L_\infty , O)$.
For a loop $\gamma \in \Omega L_\infty $, denote by ${\mathbf M}_\gamma :(S, O)\to (S, O)$ (a~germ of) the monodromy map along $\gamma $.
It is easy to see that ${\mathbf M}_\gamma $ depends only on the class $[\gamma ]\in \pi _1(L_\infty , O)$, and the map $\gamma \mapsto {\mathbf M}_\gamma $ reverses the order of multiplication,

\[
	{\mathbf M}_{\alpha \beta }={\mathbf M}_\beta \circ {\mathbf M}_\alpha .
\]

The set of all possible monodromy maps ${\mathbf M}_\gamma $, $\gamma \in \Omega L_\infty $, is called the \emph{monodromy pseudogroup} $G=G({\mathcal F})$.
The word “pseudogroup” means that there is no common domain where all elements of $G$ are defined.
We shall deal with domains of various elements of this pseudogroup, so it is not enough to deal with the corresponding \emph{group of germs} of the monodromy maps.
Nevertheless, we will follow the tradition and write “monodromy group” instead of “monodromy pseudogroup”.

\begin{remark}
	The construction of monodromy group heavily relies on the fact that the line at infinity is an algebraic leaf of ${\mathcal F}$.
	Since a~generic foliation from ${\mathcal B}_n$ has no algebraic leaves, this construction does not work for foliations from ${\mathcal B}_n$.
	This is why the analogues of many results on ${\mathcal A}_n$ are not proved for ${\mathcal B}_n$.
	Some analogues are proved for \emph{locally generic} subsets of ${\mathcal B}_n$, namely for foliations close to ${\mathcal A}_n$ \cite{LR03}, or to foliations with a rational first integral \cite{M75:en}.
	See also \autoref{rem:many-handles-Bn} below.
\end{remark}

Choose $n+1$ loops $\gamma _j\in \Omega L_\infty $, $j = 1,2,\dots ,n+1$, surrounding the points $a_j$, respectively.
Then the pseudogroup $G({\mathcal F})$ is generated by the monodromy maps ${\mathbf M}_j={\mathbf M}_{\gamma _j}$.
It is easy to see that the multipliers $\mu _j={\mathbf M}_j'(0)$ are equal to $\exp {2 \pi  i \lambda _j}$.
Recall that $\sum \lambda _j=1$, hence, $\prod \mu _j=1$.

Many theorems about properties generic foliations ${\mathcal F}\in {\mathcal A}_n$ require all singular points of a generic foliation to be complex hyperbolic in the following sense.
\begin{definition}
	\label{def:complex-hyperbolic}
	We say that a~singular point of a holomorphic vector field is \emph{complex hyperbolic}, if the linearization of the vector field at this point has two eigenvalues, and their ratio is not a real number.
\end{definition}
\subsection{Rigidity of polynomial foliations}
\label{sub:rigidity}
A generic foliation of ${\mathbb R}^2$ is \emph{structurally stable}.
More precisely, in the space of smooth vector fields on an orientable $2$-dimensional manifold, there exists an open dense subset such that each vector field from this subset is topologically conjugate to all vector fields sufficiently $C^1$-close to it, see \cites[Theorem 1]{P62}{P63}.
One can easily check that a generic polynomial foliation of ${\mathbb R}^2$ belongs to this open dense set.

For a~generic polynomial foliation of ${\mathbb C}^2$, we have the opposite property, called \emph{rigidity};
informally, a foliation sufficiently close to ${\mathcal F}$ is topologically conjugate to it only if it is affine conjugate to ${\mathcal F}$.
There are a few theorems formalizing this notion, see \cite{Il78:en,Shch84:trans,GM88,N94,VR16-utmost}.
We shall need one of them.

Fix a foliation ${\tilde {\mathcal F}}\in {\mathcal A}_n'$.
Choose $O$, $S$, and $\gamma _j$ as in \autoref{sub:mon-group}.
Fix a~neighborhood ${\mathcal U}$, ${\tilde {\mathcal F}}\in {\mathcal U}\subset {\mathcal A}_n'$, such that for foliations ${\mathcal F}\in {\mathcal U}$, their singular points at infinity are close to those of ${\tilde {\mathcal F}}$, hence the loops $\gamma _j$ still surround these points.
Following \cite[Sec.~28\textbf{E}$_1$]{IYbook}, we say that ${\mathcal F}\in {\mathcal U}$ belongs to
\begin{description}
	\item[$\Isohol({\tilde {\mathcal F}})$,]
		if there exists a~germ of~a~homeomorphism~${h:(S, O)\to (S, O)}$ conjugating each ${\mathbf M}_j({\mathcal F})$ to ${\mathbf M}_j({\tilde {\mathcal F}})$;
	\item[$\Topo({\tilde {\mathcal F}})$,]
		if there exists a~homeomorphism~${H:{\mathbb C}P^2\to {\mathbb C}P^2}$ sending the leaves of~${\mathcal F}$ to the leaves of~${\tilde {\mathcal F}}$, the line at infinity to itself, and each $\gamma _j$ to~a~curve homotopic to $\gamma _j$ in $\pi _1(L_\infty , O)$.
	\item[$\Aff({\tilde {\mathcal F}})$,]
		if it is affine equivalent to ${\tilde {\mathcal F}}$.
\end{description}
For the germs of these three sets at each point ${\tilde {\mathcal F}}\in {\mathcal A}_n'$, it is easy to see that
\[
	(\Isohol({\tilde {\mathcal F}}), {\tilde {\mathcal F}})\supseteq (\Topo({\tilde {\mathcal F}}), {\tilde {\mathcal F}})\supseteq (\Aff({\tilde {\mathcal F}}), {\tilde {\mathcal F}}),
\]
see \cite[Sec. 28\textbf{E}$_1$]{IYbook} for details.
\begin{theorem}
	[see {\cite[Theorem 28.41]{IYbook}}]
	\label{thm:rigidity-isohol}
	There exists an open dense subset ${\mathcal A}_n^R\subset {\mathcal A}_n'$ such that for each ${\tilde {\mathcal F}}\in {\mathcal A}_n^R$, we have
	\begin{equation}
		\label{eqn:rigidity-isohol}
		(\Isohol({\tilde {\mathcal F}}), {\tilde {\mathcal F}})=(\Topo({\tilde {\mathcal F}}), {\tilde {\mathcal F}})=(\Aff({\tilde {\mathcal F}}), {\tilde {\mathcal F}}).
	\end{equation}
	The exceptional set ${\mathcal A}_n\setminus {\mathcal A}_n^R$ is included by a union of a real algebraic subset $\Sigma _n$, and a real analytic subset $\Sigma _n'$ of real codimension at least~$2$.
\end{theorem}
The book \cite{IYbook} includes a detailed proof of \cite[Theorem 28.32]{IYbook} which is a~weaker version of this theorem, with larger exceptional set;
for Theorem 28.41, the authors say that the proof is analogous but requires a~stronger result on~rigidity of~groups of~germs~$({\mathbb C}, 0)\to ({\mathbb C}, 0)$.
The authors refer to~\citeauthor{Shch84:trans}, see the original article~\cite{Shch84:trans} or~a~survey~\cite{Shch06:en}, for this stronger result.
An even stronger version of this group rigidity theorem was proved by \citeauthor{N94} in \cite[Theorems 4 and 5.1]{N94}.

\begin{remark}
	Formally, Theorem 28.41 differs from the result stated above.
	Namely, it claims \emph{reasonable rigidity} of ${\tilde {\mathcal F}}$, see \cite[Definition 28.30]{IYbook},
	i.e.\ only the right equality in~\eqref{eqn:rigidity-isohol}.
	However, in the text surrounding the statement of Theorem 28.32, the authors also claim~\eqref{eqn:rigidity-isohol}, and actually prove it later.
\end{remark}
\subsection{Non-commutativity of 	monodromy groups}
As a lemma for the stronger version of the Rigidity Theorem, \citeauthor{Shch84:trans} proves that a~foliation ${\mathcal F}\in {\mathcal A}_n'\setminus \Sigma _n$ has a non-commutative unsolvable monodromy group, see \cites[Propositions 8 and 9]{Shch84:trans}[Theorem 9]{Shch06:en}.
Here $\Sigma _n$ is a union of countably many complex algebraic subsets;
it is included by a real algebraic subset, cf.~\autoref{thm:rigidity-isohol}.

Let ${\mathcal A}_n^{NC}\supset {\mathcal A}_n'\setminus \Sigma _n$ be the set of foliations having non-commutative monodromy group at infinity.
Near each foliation ${\mathcal F}\in {\mathcal A}_n'$, ${\mathcal A}_n^{NC}$ is the complement to the intersection of the analytic subsets ${\mathcal A}^C_{n,i,j}$ given by ${\mathbf M}_i\circ {\mathbf M}_j={\mathbf M}_j\circ {\mathbf M}_i$.
These subsets depend on the way we enumerate the singular points $a_j$, and on the choice of the loops $\gamma _j$.
Since the intersection of these analytic subsets is nowhere dense, \emph{at least one} of them has positive codimension.
In fact, this implies that \emph{all} of them have positive codimension.

Indeed, consider two subsets, ${\mathcal A}^C_{n,1,2}$ and ${\mathcal A}^C_{n,i,j}$.
Choose an orientation preserving homeomorphism of the sphere that sends the \emph{set} of singular points $a_l$ to itself, $O$ to $O$, $a_1$ to $a_i$, $a_2$ to $a_j$, and the homotopic classes of $\gamma _1$, $\gamma _2$ to the homotopic classes of $\gamma _i$, $\gamma _j$, respectively.
Fix an isotopy $H_t$ of the sphere joining identity to this homeomorphism such that $H_t(O)\equiv O$.
Let $a_l(t)=H_t(a_l)$ be the trajectories of the singular points under this homotopy.
Due to \autoref{rem:submersion}, there exists a curve ${\mathcal F}_t$ in ${\mathcal A}_n'$, ${\mathcal F}_0={\mathcal F}$, such that $\set{a_l(t)}$ are the singular points of ${\mathcal F}_t$, their characteristic numbers are $\lambda _l$, and their homogeneous components of lower degrees coincide with those of ${\mathcal F}$.
Then ${\mathcal F}_1={\mathcal F}$, but this isotopy sends ${\mathbf M}_1$ and ${\mathbf M}_2$ to ${\mathbf M}_i$ and ${\mathbf M}_j$, respectively.
Hence, ${\mathcal A}^C_{n,1,2}$ and ${\mathcal A}^C_{n,i,j}$ have the same dimension.
Thus \emph{all} ${\mathcal A}^C_{n,i,j}$ have positive codimension.

\begin{corollary}
	\label{non-commutative.monodromy}
	There exists an open dense subset of ${\mathcal A}_n'$ such that for each ${\mathcal F}$ from this subset none of ${\mathbf M}_i$, ${\mathbf M}_j$ commute.

	More precisely, let $U\subset {\mathcal A}_n'$ be an open subset such that we can choose $O$ and $\gamma _j$ independently on ${\mathcal F}\in U$.
	Then for ${\mathcal F}\in U$ off some analytic subset of positive codimension, none of ${\mathbf M}_i$, ${\mathbf M}_j$ commute.
\end{corollary}

\subsection{Infinite number of limit cycles}

The notion of a complex limit cycle, see \autoref{def:complex-LC}, generalizes the notion of a limit cycle of a foliation of ${\mathbb R}^2$.
Note that each isolated fixed point $w$ of some monodromy map ${\mathbf M}_\gamma \in G$ gives us a limit cycle, namely we can take the lifting of the loop $\gamma $ that starts at $w$.

\begin{definition}
	A~set of limit cycles of a~foliation is called \emph{homologically independent}, if for any leaf $L$ all the cycles located in this leaf are linearly independent in $H_1(L)$.
\end{definition}

It turns out that a generic foliation ${\mathcal F}\in {\mathcal A}_n$, $n\geq 2$, possesses countably many homologically independent limit cycles.
This was proved in \cite{Il78:en} for a full measure subset of ${\mathcal A}_n$, $n\geq 2$, then in \cite{SRO98} for an open dense subset of ${\mathcal A}_n$, $n\geq 3$.
Recently, we proved this statement for an open dense subset of ${\mathcal A}_n$, $n\geq 2$, and for its open neighborhood in ${\mathcal B}_{n+1}$, see \cite{GK-BLC,GK-CLC}.

The proofs in \cite{Il78:en,SRO98,GK-BLC} rely on the fact that for a generic foliation ${\mathcal F}\in {\mathcal A}_n$, the fixed points of the elements of the monodromy pseudogroup are dense in a small neighborhood of the origin.
Each of the fixed points corresponds to a  limit cycle;
the proof of their homological independence is not trivial.

We provide a refined version of Ilyashenko's lemma \cite[Lemma 3*]{Il78:en}.

\begin{lemma}
	\label{lem:LCs-dense}
	Suppose that two monodromy maps ${\mathbf M}_1$ and ${\mathbf M}_2$ do not commute, and their multipliers satisfy

	\begin{itemize}
		\item $0<|\mu _1|<1, |\mu _2|>1$;
		\item the multiplicative semigroup generated by $\mu _1$ and $\mu _2$ is dense in ${\mathbb C}^*$.
	\end{itemize}

	Then the set of hyperbolic fixed points of compositions of the form ${\mathbf M}_1^{-s}{\mathbf M}_2^{t}{\mathbf M}_1^{r+s}{\mathbf M}_2$, $r, s, t\in {\mathbb N}$, is dense in a~small neighborhood of the origin.
\end{lemma}

Though the statement of \cite[Lemma 3*]{Il78:en} slightly differs from \autoref{lem:LCs-dense}, Ilyashenko's proof actually implies this lemma.
Nevertheless, we shall provide the proof, both for completeness, and because we shall need a refined version of an auxiliary lemma, see \autoref{cor:approximate-linear} below.
See also \cite[Lemmas 3 and 4]{V06:en}.

Since $0<|\mu _1|<1$, we can linearize ${\mathbf M}_1$ in a small neighborhood of the origin, see \cite{Koen84,Sch70,Sch71}.
We shall work in this linearizing chart.
The next lemma and its corollary show that we can choose $r, s, t\in {\mathbb N}$ so that ${\mathbf M}_1^{-s}{\mathbf M}_2^{t}{\mathbf M}_1^{r+s}$ approximates any prescribed linear map.

\begin{lemma}
	\label{lem:approximate-derivative}
	Consider an analytic map $g:U\to U$, $g(0)=0$, defined in a neighborhood $U$ of the origin,
	and a complex number $\mu $, $0<|\mu |<1$.
	Then $\mu ^{-s}g(\mu ^sz)$, $s\in {\mathbb N}$, converges to $g'(0)z$ uniformly in $z$ on any disc $|z|\leq R$ as $s\to \infty $.
	Moreover, if $g$ analytically depends on a parameter $\eps \in V\subset {\mathbb C}^n$ then the convergence above is uniform in $\eps $ on compact sets.
\end{lemma}
\begin{proof}
	Consider a family of maps $g_\eps :U\to U$, $g_\eps (0)=0$, $\eps \in K\subset {\mathbb C}^n$, $K$ is a compact set.
	Put $z_s=\mu ^sz$, then for $|z|\leq R$ we have
	\[
		\left| \mu ^{-s}g_\eps (\mu ^sz)-g_\eps '(0)z \right|
			=|z|\cdot \left| \frac{g_\eps (z_s)}{z_s}-g_\eps '(0) \right|
			\leq R\left| \frac{g_\eps (z_s)}{z_s}-g_\eps '(0) \right|.
	\]
	Since $g_\eps (z)$ is analytic both in $z$ and $\eps $, the latter fraction tends to $0$ as $z_s\to 0$ uniformly in $\eps \in K$.
\end{proof}

\begin{corollary}
	\label{cor:approximate-linear}
	Let $g:({\mathbb C}, 0)\to ({\mathbb C}, 0)$ be an expanding analytic germ, $|g'(0)|>1$.
	Let $\mu $, $|\mu |<1$, be a number such that the multiplicative semigroup generated by $\mu $ and $g'(0)$ is dense in ${\mathbb C}^*$.
	Then for each $\nu \in {\mathbb C}^*$ and $R>0$, the linear map $z\mapsto \nu z$ can be uniformly approximated in the disc $|z|\leq R$ by a map of the form $z\mapsto \mu ^{-s}g^t(\mu ^{r+s}z)$, $r,s,t\in {\mathbb N}$.

	If $g$ analytically depends on a parameter $\eps \in V\subset {\mathbb C}^n$ so that $g'(0)=\const$, then the approximation can be made uniform in $\eps $ on any compact subset $K\subset V$.
\end{corollary}
\begin{proof}
	Since the multiplicative semigroup generated by $\mu $ and $g'(0)$ is dense in ${\mathbb C}^*$, we can choose $t, r\in {\mathbb N}$ such that $g'(0)^t\mu ^r$ is very close to $\nu $.
	Then \autoref{lem:approximate-derivative} applied to the map $z\mapsto g^t(\mu ^rz)$ completes the proof.
\end{proof}

\begin{proof}
	[Proof of \autoref{lem:LCs-dense}]
	\phantomsection\label{proof:lem:LCs-dense}
	As we mentioned above, we will work in a linearizing chart for ${\mathbf M}_1$.
	Choose a disc $D_R=\set{z|\relax |z|\leq R}$ in the domain of ${\mathbf M}_1$, and consider an open subset $U\subset D_R$.
	Let us find $r, s, t\in {\mathbb N}$ such that the map ${\mathbf M}_1^{-s}{\mathbf M}_2^{t}{\mathbf M}_1^{r+s}{\mathbf M}_2$ has a hyperbolic fixed point in $U$.

	Fix a point $w\in U$, and consider the map $h$ given by $h(z)=\frac{w}{{\mathbf M}_2(w)}{\mathbf M}_2(z)$.
	Clearly, $w$ is a fixed point of $h$, and $h'(w)=\frac{w}{{\mathbf M}_2(w)}{\mathbf M}_2'(w)$.
	Note that we can choose $w\in U$ so that $|h'(w)|\neq 1$.
	Indeed, otherwise $\left|\frac{w}{{\mathbf M}_2(w)}{\mathbf M}_2'(w)\right|\equiv 1$ in $U$, hence ${\mathbf M}_2'(w)=Cw^{-1}{\mathbf M}_2(w)$, where $C$ is a constant with $|C|=1$.
	Thus ${\mathbf M}_2(w)=C_1w^C$.
	Since ${\mathbf M}_2$ is holomorphic at the origin, we have $C=1$ and ${\mathbf M}_2(w)=C_1w$, which contradicts ${\mathbf M}_1\circ {\mathbf M}_2\neq {\mathbf M}_2\circ {\mathbf M}_1$.

	So, we have $w\in U$, $h(w)=w$, $|h'(w)|\neq 1$.
	Recall that in our chart ${\mathbf M}_1(z)=\mu _1z$, hence \autoref{cor:approximate-linear} provides us with $r, s, t\in {\mathbb N}$ such that the map ${\mathbf M}_1^{-s}{\mathbf M}_2^{t}{\mathbf M}_1^{r+s}$ uniformly approximates the linear map $z\mapsto \frac{w}{{\mathbf M}_2(w)} z$ in $D_R\supset U$.
	Thus the map ${\mathbf M}_1^{-s}{\mathbf M}_2^{t}{\mathbf M}_1^{r+s}{\mathbf M}_2$ approximates $h$, hence it has a hyperbolic fixed point near $w$.
\end{proof}
\subsection{Volk's Theorem}
\label{sub:volk}

\subsubsection{Statement of the theorem}
In \cite{V06:en} D.~Volk proves that foliations with separatrix connections are dense in ${\mathcal A}_n$.
Actually, his arguments work in a~more general setting.

\begin{theorem}
	\label{thm:volk-holo-maps}
	Let ${\tilde {\mathcal F}}$ be a polynomial foliation of degree $n\geq 2$.
	Let $A, B$ be holomorphic maps of a neighborhood of ${\tilde {\mathcal F}}$ in ${\mathcal A}_n$ to ${\mathbb C}^2$ such that $A({\tilde {\mathcal F}})$ and $B({\tilde {\mathcal F}})$ are non-singular points of ${\tilde {\mathcal F}}$.
	Then there exists ${\mathcal F}$ arbitrarily close to ${\tilde {\mathcal F}}$ such that the points $A({\mathcal F})$ and $B({\mathcal F})$ belong to the same leaf of ${\mathcal F}$.
\end{theorem}

The original Volk's Theorem follows from this theorem if $A({\mathcal F})$ and $B({\mathcal F})$ belong to separatrices of two different singular points of ${\mathcal F}$.
However, we shall need a more precise statement.

Take a~compact analytic submanifold ${\mathcal M}\subset {\mathcal A}_n^R$ with boundary.
Suppose that

\begin{itemize}
	\item $\dim {\mathcal M}>\dim\Aff({\mathbb C}^2)=6$;
	\item $\mu _1=\const$ and $\mu _2=\const$ on ${\mathcal M}$;
	\item $|\mu _1|<1$ and $|\mu _2|<1$;
	\item the multiplicative semigroup generated by $\mu _1$ and $\mu _2^{-1}$ is dense in ${\mathbb C}^*$.
\end{itemize}

Fix a foliation ${\tilde {\mathcal F}}\in {\mathcal M}\setminus \partial {\mathcal M}$, its non-singular point $O\in L_\infty $, and a cross-section $S=\set{v=\const}$, $S\cap L_\infty =\set{O}$.
Let $z_{\mathcal F}:(S, O)\to ({\mathbb C}, 0)$ be a linearizing chart for~${\mathbf M}_1$ analytically depending on ${\mathcal F}$.
In the rest of this section, ${\mathbf M}_j$, $A({\mathcal F})$, $B({\mathcal F})$ etc.\ are written in the corresponding chart $z_{\mathcal F}$.
Denote by ${\tilde {\mathbf M}}_j$ the $j$-th generator of $G({\tilde {\mathcal F}})$ written in $z_{\tilde {\mathcal F}}$.
In particular, ${\mathbf M}_1(z)=\mu _1z$, ${\tilde {\mathbf M}}_1({\tilde z})=\mu _1{\tilde z}$.
Let us rescale $z_{\mathcal F}$, choose its representative, and diminish $S$ and ${\mathcal M}$ so that for all ${\mathcal F}\in {\mathcal M}$, $S$ is in the domain of $z_{\mathcal F}$, and $z_{\mathcal F}(S)$ includes the unit disc.

Let $A, B:{\mathcal M}\to S$ be two holomorphic functions such that $0<|A({\tilde {\mathcal F}})|<1$ and $0<|B({\tilde {\mathcal F}})|<1$.

\begin{theorem}
	\label{volk.details}
	In the setting introduced above, for any $\eps >0$ there exist a~neighborhood~${\mathcal U}$, ${\tilde {\mathcal F}}\in {\mathcal U}\subset {\mathcal M}$, and a~loop~$\gamma \in \Omega L_\infty $ such that for ${\mathcal F}\in {\mathcal U}$ we have
	\begin{itemize}
		\item $A({\mathcal F})$ belongs to the domain of~${\mathbf M}_\gamma $;
		\item ${\mathbf M}_\gamma (A({\mathcal F}))=B({\mathcal F})$ defines a non-empty submanifold of~codimension one;
		\item the derivative~${\mathbf M}_\gamma '(A({\mathcal F}))$ is $\eps $-close to $\frac{B({\tilde {\mathcal F}})}{A({\tilde {\mathcal F}})}$.
	\end{itemize}

	The loop $\gamma $ can be constructed as the concatenation $\gamma =\alpha \beta $ of the following loops.
	There exists an index $i$ such that for each sufficiently large $p$ we can choose either $\alpha =\gamma _1^p\gamma _i$ or $\alpha =\gamma _1^p$.
	After $\alpha $ is fixed, there exists a~triple of~arbitrarily large numbers $(r, s, t)\in {\mathbb N}^3$ such that we can take $\beta =\gamma _1^{r+s}\gamma _2^{-t}\gamma _1^{-s}$.
\end{theorem}

\begin{remark}
	The most important part of this theorem is the second assertion.
	We include technical details like the estimate on ${\mathbf M}_\gamma '(A({\mathcal F}))$ and the explicit construction of $\gamma $, because we will need them later in \autoref{sec:many-handles}.
\end{remark}
\autoref{thm:volk-holo-maps} follows from \autoref{volk.details}, and density of the leaves of a generic foliation ${\mathcal F}\in {\mathcal A}_n$.

The rest of this section is devoted to the proof of \autoref{volk.details}.
First, we shall prove it in a particular case, if the ratio $\frac{B({\mathcal F})}{A({\mathcal F})}$ is a non-constant function of ${\mathcal F}$.
This case with $\gamma =\beta $ instead of $\gamma =\alpha \beta $ easily follows from \autoref{cor:approximate-linear}, see \autoref{volk.details.submanifold} below.

Next, we shall choose $\alpha $ such that $\frac{B({\mathcal F})}{{\mathbf M}_\alpha (A({\mathcal F}))}$ is a non-constant function of ${\mathcal F}$, and use ${\mathbf M}_\alpha (A({\mathcal F}))$ as $A({\mathcal F})$ in \autoref{volk.details.submanifold}.
The choice of $\alpha $ heavily relies on the Rigidity \autoref{thm:rigidity-isohol}, see \autoref{volk.details.non.proportional} below.

\subsubsection{Proof in a particular case}
\begin{lemma}
	\label{volk.details.submanifold}
	In the setting of~\autoref{volk.details}, suppose that the ratio $\frac{B({\mathcal F})}{A({\mathcal F})}$, ${\mathcal F}\in ({\mathcal M}, {\tilde {\mathcal F}})$, is a~non-constant function of ${\mathcal F}$.
	Then the assertions of~\autoref{volk.details} hold for~$\gamma =\beta $ instead of~$\gamma =\alpha \beta $.
\end{lemma}
\begin{proof}
	Put $\xi =\dfrac{B({\tilde {\mathcal F}})}{A({\tilde {\mathcal F}})}$, then the equality $B({\mathcal F})=\xi A({\mathcal F})$ defines a non-empty codimension one analytic submanifold of ${\mathcal M}$.

	Approximate the linear map $z\mapsto \xi z$ by a~map of the form ${\mathbf M}_\beta ={\mathbf M}_1^{-s}{\mathbf M}_2^{-t}{\mathbf M}_1^{r+s}$, see~\autoref{cor:approximate-linear}.
	This approximation is uniform in ${\mathcal F}\in {\mathcal M}$.
	For ${\mathcal F}$ sufficiently close to ${\tilde {\mathcal F}}$, $A({\mathcal F})$ belongs to the domain of ${\mathbf M}_\beta $, and ${\mathbf M}_\beta (A({\mathcal F}))=B({\mathcal F})$ defines a codimension-one nonempty submanifold.

	Finally, due to the Cauchy estimates, ${\mathbf M}_\beta $ approximates the multiplication by $\xi $ in $C^\omega $ topology, hence ${\mathbf M}_\beta '(A({\mathcal F}))$ is close to~$\xi $.
\end{proof}

\subsubsection{Reduction to the particular case}
\begin{lemma}
	\label{volk.details.non.proportional}
	In the setting of \autoref{volk.details}, we can find an index $i$ such that for each sufficiently large~$p$ either for $\alpha =\gamma _1^p\gamma _i$ or for $\alpha =\gamma _1^p$,
	\begin{itemize}
		\item $A({\mathcal F})$ belongs to the domain of ${\mathbf M}_\alpha $ for all ${\mathcal F}\in {\mathcal M}$;
		\item the ratio $\dfrac{B({\mathcal F})}{{\mathbf M}_\alpha (A({\mathcal F}))}$ is a non-constant function of~${\mathcal F}\in {\mathcal M}$;
		\item the ratio ${\mathbf M}_\alpha '(A({\mathcal F}))\div \dfrac{{\mathbf M}_\alpha (A({\mathcal F}))}{A({\mathcal F})}$ tends to~$1$ uniformly in ${\mathcal F}\in {\mathcal M}$.
	\end{itemize}
\end{lemma}

\begin{remark}
	\autoref{volk.details.non.proportional} is a refined version of the union of \cite[Lemmas 6 and 7]{V06:en}.
	The proof of \cite[Lemma 6]{V06:en} deals separately with $n\geq 3$ and $n=2$;
	unfortunately, the proof for the case $n=2$ has a gap.
	We give another proof which works for all $n\geq 2$.
\end{remark}

\begin{proof}
	[Proof of \autoref{volk.details.non.proportional}]
	Let us prove the assertions one by one.

	\paragraph{Domains of definition.}
	Recall that $|A({\mathcal F})|<1$ and $|\mu _1|<1$, hence $A({\mathcal F})$ belongs to the domain of ${\mathbf M}_1^p$ for all $p$, and~$\left|{\mathbf M}_1^p(A({\mathcal F}))\right|<\mu _1^p\to 0$ as $p\to \infty $.
	Therefore, for $p$ large enough, ${\mathbf M}_1^p(A({\mathcal F}))$ belongs to the domains of~all maps ${\mathbf M}_j$.
	So, $A({\mathcal F})$ belongs to the domains of all the maps of the form ${\mathbf M}_1^p$ and ${\mathbf M}_j\circ {\mathbf M}_1^p$.

	\paragraph{Non-constant ratio.}
	Note that
	\[
		\frac{{\mathbf M}_{\gamma _1^p\gamma _i}(A({\mathcal F}))}{{\mathbf M}_{\gamma _1^p}(A({\mathcal F}))}=\frac{B({\mathcal F})}{{\mathbf M}_{\gamma _1^p}(A({\mathcal F}))}\div \frac{B({\mathcal F})}{{\mathbf M}_{\gamma _1^p\gamma _i}(A({\mathcal F}))},
	\]
	hence it is enough to prove that the left hand side is a non-constant function of ${\mathcal F}$ for $p$ large enough.
	Put $A_p({\mathcal F})={\mathbf M}_{\gamma _1^p}(A({\mathcal F}))=\mu _1^pA({\mathcal F})$, then $A_p({\mathcal F})\to 0$ as $p\to \infty $, and
	\[
		\frac{{\mathbf M}_{\gamma _1^p\gamma _i}(A({\mathcal F}))}{{\mathbf M}_{\gamma _1^p}(A({\mathcal F}))}=\frac{{\mathbf M}_i(A_p({\mathcal F}))}{A_p({\mathcal F})}.
	\]

	From now on, let ${\mathcal F}\in {\mathcal M}$ be a foliation close to ${\tilde {\mathcal F}}$ that we will choose later.
	Then it is enough to find ${\mathcal F}\in {\mathcal M}$ and $i$ such that
	\begin{align*}
		\frac{{\tilde {\mathbf M}}_i(A_p({\tilde {\mathcal F}}))}{A_p({\tilde {\mathcal F}})}&\neq \frac{{\mathbf M}_i(A_p({\mathcal F}))}{A_p({\mathcal F})},
		\intertext{or equivalently}
		{\tilde {\mathbf M}}_i\left(\frac{A({\tilde {\mathcal F}})}{A({\mathcal F})}A_p({\mathcal F})\right)&\neq \frac{A({\tilde {\mathcal F}})}{A({\mathcal F})}{\mathbf M}_i(A_p({\mathcal F})).
	\end{align*}
	Since $A_p({\mathcal F})\to 0$ as $p\to \infty $, it is enough to find ${\mathcal F}\in {\mathcal M}$ close to ${\tilde {\mathcal F}}$ such that the map given by $z\mapsto \frac{A({\tilde {\mathcal F}})}{A({\mathcal F})}z$ in the charts $(z_{\mathcal F}, z_{\tilde {\mathcal F}})$ does not conjugate $G({\mathcal F})$ to $G({\tilde {\mathcal F}})$.
	Due to \autoref{thm:rigidity-isohol}, this holds for any ${\mathcal F}\in ({\mathcal M}, {\tilde {\mathcal F}})$ which is not affine equivalent to ${\tilde {\mathcal F}}$.
	Recall that $\dim {\mathcal M}>\dim\Aff({\mathbb C}^2)$, hence such ${\mathcal F}$ exists.

	\paragraph{Estimate on the derivative.}
	If $\alpha =\gamma _1^p$, then the ratio ${\mathbf M}_\alpha '(A({\mathcal F}))\div \dfrac{{\mathbf M}_\alpha (A({\mathcal F}))}{A({\mathcal F})}$ equals one.
	Let us consider the remaining case $\alpha =\gamma _1^p\gamma _i$.
	In this case
	\[
		{\mathbf M}_\alpha '(A({\mathcal F}))\div \frac{{\mathbf M}_\alpha (A({\mathcal F}))}{A({\mathcal F})}={\mathbf M}_i'(A_p({\mathcal F}))\div \frac{{\mathbf M}_i(A_p({\mathcal F}))}{A_p({\mathcal F})}.
	\]
	Note that both ${\mathbf M}_i'(z)$ and $\frac{{\mathbf M}_i(z)}{z}$ tend to ${\mathbf M}_i'(0)$ as $z\to 0$ uniformly in ${\mathcal F}$, hence the right hand side of the equality above tends to one as $p\to \infty $ uniformly in ${\mathcal F}\in {\mathcal M}$.
\end{proof}

\subsubsection{Deducing \autoref{volk.details} from the lemmas}
\begin{proof}
	[Proof of \autoref{volk.details}]
	Choose $\alpha $ as in \autoref{volk.details.non.proportional}, then choose a neighborhood ${\mathcal U}$, ${\tilde {\mathcal F}}\in {\mathcal U}\subset {\mathcal M}$, such that $\dfrac{{\mathbf M}_\alpha (A({\mathcal F}))}{A({\mathcal F})}\approx \dfrac{{\tilde {\mathbf M}}_\alpha (A({\tilde {\mathcal F}}))}{A({\tilde {\mathcal F}})}$ for ${\mathcal F}\in {\mathcal U}$.
	Then the ratio ${\mathbf M}_\alpha '(A({\mathcal F}))\div \dfrac{{\tilde {\mathbf M}}_\alpha (A({\tilde {\mathcal F}}))}{A({\tilde {\mathcal F}})}$ is close to~$1$ for ${\mathcal F}\in {\mathcal U}$.

	Apply \autoref{volk.details.submanifold} to ${\mathbf M}_\alpha \circ A$, $B$ and ${\mathcal U}$ instead of $A$, $B$ and ${\mathcal M}$.
	Then the resulting curves $\alpha $, $\beta $ satisfy the assertions of \autoref{volk.details} in the new domain ${\mathcal U}$ provided by~\autoref{volk.details.submanifold}.
	For the first two assertion this is clear, and the last one immediately follows from the chain rule,
	\[
		{\mathbf M}_\gamma '(A({\mathcal F}))
			={\mathbf M}_\beta '({\mathbf M}_\alpha (A({\mathcal F}))){\mathbf M}_\alpha '(A({\mathcal F}))
			\approx \frac{B({\tilde {\mathcal F}})}{{\tilde {\mathbf M}}_\alpha (A({\tilde {\mathcal F}}))}\times \frac{{\tilde {\mathbf M}}_\alpha (A({\tilde {\mathcal F}}))}{A({\tilde {\mathcal F}})}
			=\frac{B({\tilde {\mathcal F}})}{A({\tilde {\mathcal F}})}.
	\]
\end{proof}

\subsection{Intersections with lines}
In the proof of \autoref{infinite.genus} in \autoref{sec:inf-genus}, we shall need to prove that a generic leaf of a generic foliation from ${\mathcal A}_n^{sym}$ or ${\mathcal B}_n^{sym}$ intersects the line $y=0$ in infinitely many points.
The proof will be based on the following two statements.
First, we use a theorem due to Jouanolou to estimate the number of algebraic leaves.

\begin{theorem}
	[{\cite[Theoréme 3.3, p. 102]{J79}}]
	\label{few-algebraic-leaves}

	If a polynomial foliation ${\mathcal F}\in {\mathcal B}_n$ has at least $\frac 12 n(n - 1) + 2$ algebraic irreducible invariant curves, then it has a rational first integral.
\end{theorem}

Then we prove that a non-algebraic leaf intersects a generic line in infinitely many points.

\begin{lemma}
	\label{infinite-intersection}
	Consider a polynomial foliation ${\mathcal F}$ of ${\mathbb C}P^2$, a non-algebraic leaf $L$ and a~line $T\subset {\mathbb C}P^2$ not passing through singular points of ${\mathcal F}$.
	Then $\#(L\cap T)=\infty $.
\end{lemma}

The proof is completely analogous to the proof of \cite[Lemma 28.10]{IYbook}.
This lemma states that a~non-algebraic leaf of a~foliation ${\mathcal F}\in {\mathcal A}_n'$ cannot approach the line at infinity only along the separatrices of singular points.
However, we repeat the proof here for completeness.

\begin{proof}
	Suppose the contrary, i.\ e.~$L$ is not algebraic and $\#(L\cap T)<\infty $.
	Since $T$ contains no singular points of ${\mathcal F}$, $T$ is not a leaf of ${\mathcal F}$, and the number of intersections $\#(L\cap T)$ counted with multiplicities is locally a constant.
	After a small perturbation of $T$, we may and will assume that all intersections of $L$ and $T$ are transverse.

	Make a projective coordinate change such that $T$ is mapped to the line at infinity $\{u=0\}$, and the point $v=\infty $ of the line at infinity does not belong to $L$.
	It is easy to show that $L$ cannot be bounded, hence it must intersect $\{u=0\}$ in at least one point.

	Suppose that the leaf $L$ is given by $y=\varphi _j(x)$, $j=1,2,\dots ,k$, $x\in ({\mathbb C}, \infty )$, in neighborhoods of $k$ points of $L\cap T$.
	Note that each $\varphi _j$ has a linear growth at infinity.
	Consider the product $\prod _{j=1}^k (y - \varphi _j(x))$;
	this is a polynomial in $y$, with symmetric functions $\sigma _1 = \sum _{j=1}^k \varphi _j$, $\sigma _2=\sum _{1\leq j<l\leq k} \varphi _j \varphi _l,$ \dots , $\sigma _k =\prod _{j=1}^k \varphi _j$ as coefficients.

	Let $P_i$ be projections of finite singularities of ${\mathcal F}$ to $x$-plane.
	It is possible to extend $\sigma _j$ holomorphically to ${\mathbb C}\setminus \set{P_1,\dots }$ by the symmetric combinations of intersections $L\cap \set{x = c}$, with multiplicities.
	Indeed, the number of these intersections stays locally the same, thus equals $k$ for any $c$.
	The intersections depend holomorphically on $c$ and stay bounded, otherwise the leaf $L$ would approach the line at infinity along $x = \const$, hence the point $v=\infty , u=0$ would belong to~$L$.

	Since $\sigma _j$ are bounded in any compact, $P_i$ are removable singularities of $\sigma _j$.

	So, the symmetric combinations of $\varphi _j$ extend holomorphically to ${\mathbb C}$ and have a polynomial growth at infinity.
	Thus they are polynomials in $x$, and the function $F=\prod _{j=1}^k (y - \varphi _j(x))$ is a polynomial in $x,y$.
	Hence $F=0$ is a polynomial equation defining the leaf $L$, and $L$ is algebraic.
	The contradiction shows that $\#(L\cap T)=\infty $ for a~non-algebraic $L$.
\end{proof}

\section{A leaf with many handles}
\label{sec:many-handles}

In this section we shall prove \autoref{positive.genus}.

Consider an~open subset ${\mathcal U}\subset {\mathcal A}_n$.
Shrinking ${\mathcal U}$ if necessary, we may and will assume that

\begin{itemize}
	\item ${\mathcal U}\subset {\mathcal A}_n^R$ (due to \autoref{thm:rigidity-isohol}, ${\mathcal A}_n^R$ is open and dense);
	\item $\Im \lambda _j({\mathcal F})\neq 0$ for ${\mathcal F}\in {\mathcal U}$ and $j=1,\dots ,n+1$;
	\item one can enumerate the singular points $a_1, \dots , a_{n+1}$ at the line at infinity so that $a_i$ depend analytically on ${\mathcal F}\in {\mathcal U}$;
	\item ranges of $a_i({\mathcal U})$ are small enough so that we can and shall fix a point $O$ and paths $\gamma _i$ as in \autoref{sub:rigidity} independently on ${\mathcal F}\in {\mathcal U}$;
	\item for ${\mathcal F}\in {\mathcal U}$, none of $\mu _i$ belongs to the unit circle.
	\item ${\mathbf M}_i\circ {\mathbf M}_j\neq {\mathbf M}_j\circ {\mathbf M}_i$ for ${\mathcal F}\in {\mathcal U}$ and $i\neq j$ (due to \autoref{non-commutative.monodromy}).
\end{itemize}

Note that passing to the complex conjugate coordinates $({\overline x}, {\overline y})$ in ${\mathbb C}^2$ replaces all $|\mu _j|=|\exp(2\pi i\lambda _j)|=\exp(-2\pi \Im \lambda _j)$ by $\exp(-2\pi \Im ({\overline \lambda }_j))=|\mu _j|^{-1}$.
Therefore, we can assume that at least two of $|\mu _j|$ are less than one (otherwise we just pass to the conjugate coordinates).
Recall that $\prod \mu _j=1$, hence at least one of $|\mu _j|$ is greater than one.
Let us reenumerate the singularities at the line at infinity so that the multipliers satisfy

\begin{itemize}
	\item $|\mu _1|<1$, $|\mu _2|>1$ and $|\mu _3|<1$.
\end{itemize}

Let ${\mathcal M}_0\subset {\mathcal U}$ be a~non-empty submanifold given by $\mu _1=\const$, $\mu _2=\const$ and $\mu _3=\const$.
Slightly perturbing constants in these equations, assume that for ${\mathcal F}\in {\mathcal M}_0$,

\begin{itemize}
	\item the multiplicative semigroup generated by $\mu _1$ and $\mu _2$ is~dense in ${\mathbb C}^*$;
	\item the multiplicative semigroup generated by $\mu _1^{-1}$ and $\mu _3$ is~dense in ${\mathbb C}^*$.
\end{itemize}

For $n=2$,  $\codim {\mathcal M}_0=2$: the equations $\mu _1=\const$, $\mu _2=\const$ imply that $\mu _3=\const$ since $\mu _1\mu _2\mu _3=1$.
For $n\geq 3$, $\codim {\mathcal M}_0=3$.
The following lemma is a key step in the proof of \autoref{positive.genus}.

\begin{lemma}
	\label{one.handle}
	Let ${\mathcal M}_0$ and $\mu _i$, $i=1, 2, 3$ be as above.
	Let ${\mathcal M}\subset {\mathcal M}_0$ be an analytic submanifold of dimension at least $7$.
	Then for any $\eps >0$ there exists a submanifold ${\mathcal M}'\subset {\mathcal M}$ of codimension one such that each ${\mathcal F}\in {\mathcal M}'$ has a leaf with a handle $\eps $-close to $L_\infty $.

	More precisely, for ${\mathcal F}\in {\mathcal M}'$ there exist two curves $\gamma ^{(1)},\gamma ^{(2)} \subset  L_\infty $ such that ${\mathbf M}_{\gamma ^{(1)}}$ and ${\mathbf M}_{\gamma ^{(2)}}$ have a common hyperbolic fixed point $B=B({\mathcal F})\in S$, the lifts $c_1,c_2$ of the curves $\gamma ^{(1)}, \gamma ^{(2)}$ starting from $B$ intersect transversally at exactly one point, and $c_1$ and $c_2$ are included by $\eps $-neighborhood of the line at infinity.
\end{lemma}

We shall postpone the proof of this lemma till the end of this section.
Now let us deduce \autoref{positive.genus} from this lemma.
First, we obtain many handles on different leaves.

\begin{corollary}
	\label{many.handles}

	For each $0\leq g\leq \dim {\mathcal M}_0-6$, there exists an analytic submanifold ${\mathcal M}_g\subset {\mathcal M}_0$ of codimension at most $g$ such that the leaves of each ${\mathcal F}\in {\mathcal M}_g$ possess $g$ handles (possibly on different leaves of ${\mathcal F}$) with hyperbolic generating cycles $(c_1, c_2), (c_3, c_4), \dots , (c_{2g-1}, c_{2g})$.
	The  generators of different handles do not intersect (even if they are located in the same leaf), and $c_{2j-1}$ transversally intersects $c_{2j}$ at a point $B_j\in S$.
\end{corollary}

\begin{proof}
	Let us prove the assertion by induction.
	For $g=0$, we just take ${\mathcal M}_0$.
	Suppose that we already have ${\mathcal M}_g$, $g\leq \dim {\mathcal M}_0-7$.
	Then $\dim {\mathcal M}_g\geq 7$.
	Using \autoref{one.handle}, we get a submanifold ${\mathcal M}_{g+1}\subset {\mathcal M}_g$ of codimension $1$ such that each ${\mathcal F}\in {\mathcal M}_{g+1}$ possesses a handle generated by $(c_{2g+1}, c_{2g+2})$ which is closer to $L_\infty $ than all the loops guaranteed by ${\mathcal M}_g$.
	Hence, ${\mathcal M}_{g+1}$ satisfies the assertion of this corollary.
	This completes the proof.
\end{proof}

\begin{proof}
	[Proof of \autoref{positive.genus} modulo \autoref{one.handle}]
	Let us apply the previous corollary to
	\[
		g=\frac{(n+1)(n+2)}2-4.
	\]
	This is possible since for $n\geq 2$ we have

	\[
		g=\frac{(n+1)(n+2)}2-4\leq (n+1)(n+2)-10\leq \dim {\mathcal M}_0-6.
	\]

	Note that the cycles $c_{2j-1}$ and $c_{2j}$ correspond to hyperbolic fixed points of the germs of monodromy maps ${\mathbf M}_{c_{2j-1}}, {\mathbf M}_{c_{2j}}:(S, B_j)\to (S, B_j)$.
	Hence they survive under a small perturbation.
	The submanifold ${\mathcal M}_g\subset {\mathcal M}_0$ is defined by $g$ equations of the form “Hyperbolic fixed points of ${\mathbf M}_{c_{2i-1}}$ and ${\mathbf M}_{c_{2i}}$ coincide”.

	Let ${\tilde {\mathcal M}}_0\subset {\mathcal U}$, ${\tilde {\mathcal M}}_0\supset {\mathcal M}_0$ be the submanifold defined by $\mu _1=\const$, $\mu _2=\const$.
	It is easy to see that $\dim {\tilde {\mathcal M}}_0=(n+1)(n+2)-3$ and ${\mathcal M}_g$ extends to a submanifold ${\tilde {\mathcal M}}_g\subset {\tilde {\mathcal M}}_0$ given by the same $g$ equations.
	Thus ${\tilde {\mathcal M}}_g$ has codimension $g$ in ${\tilde {\mathcal M}}_0$,

	\[
		\dim {\tilde {\mathcal M}}_g=\dim {\tilde {\mathcal M}}_0-g=\frac{(n+1)(n+2)}2+1=(g-1)+6
	\]

	The leaves of the foliations from ${\tilde {\mathcal M}}_g$ also have at least $g$ handles (possibly these handles are on the different leaves).
	Indeed, after a small perturbation of a foliation ${\mathcal F}\in {\mathcal M}_g$ inside ${\tilde {\mathcal M}}_g$, all cycles $c_j$ survive and still intersect transversally at the points $B_j=B_j({\mathcal F})$.

	Now, let us apply \autoref{volk.details} $(g-1)$ times to the monodromy maps ${\mathbf M}_1,{\mathbf M}_2$ and the points $B_1,B_2, \dots , B_g$; we obtain a~$6$-dimensional submanifold ${\hat {\mathcal M}}\subset {\tilde {\mathcal M}}_0$ such that for each ${\mathcal F}\in {\hat {\mathcal M}}$ all points $B_j$ are located in the same leaf, thus all generating cycles $c_j$ are located in the same leaf.
\end{proof}
\begin{remark}
	\label{rem:many-handles-Bn}
	The natural question is, whether the analogue of \autoref{positive.genus} holds true in some open subset in ${\mathcal B}_n$.
	This question is open.
	Here are some arguments towards its solution.

	Consider a small neighborhood of ${\mathcal A}_n$ in ${\mathcal B}_{n+1}$.
	The maps ${\mathbf M}_j$ survive in this neighborhood, though they do not preserve the origin anymore.
	The same equations used to define ${\hat {\mathcal M}}$ still define a submanifold in ${\mathcal B}_{n+1}$ such that each foliation from this submanifold possesses a leaf with at least $\frac{(n+1)(n+2)}2-4$ handles.

	We can repeat this construction for every ${\hat {\mathcal M}}$ produced by the proof of \autoref{positive.genus}, and obtain countably many submanifolds of ${\mathcal B}_{n+1}$.
	If the union of these submanifolds is dense in a neighborhood of ${\mathcal A}_n$ in ${\mathcal B}_{n+1}$, then the result generalizes to ${\mathcal B}_{n+1}$;
	however we do not know whether this statement holds true.
\end{remark}

Now let us prove \autoref{one.handle}.

\begin{proof}
	[Proof of \autoref{one.handle}]
	As in \autoref{volk.details}, let us choose a linearizing chart $z_{\mathcal F}$ of ${\mathbf M}_3$ analytically depending on ${\mathcal F}$.
	We will use this coordinate a few times, though most computations will be done in the $(u, v)$-coordinates.

	\paragraph{Construction of the first cycle.}
	Due to \autoref{lem:LCs-dense}, there exist $k$, $l$ and $m$ such that ${\mathbf M}_1^{-k} {\mathbf M}_2^{m} {\mathbf M}_1^{k+l}{\mathbf M}_2$ has a~hyperbolic fixed point $B({\mathcal F})$ near the origin.
	We require some additional conditions on $B({\mathcal F})$.
	More precisely, we proceed in three steps.
	First we choose a~domain~$U\subset S$, $O\notin {\overline U}$, for $B({\mathcal F})$ sufficiently close to~$O$, so that
	\begin{equation}
		\label{eqn:B-ineq-M1-M2}
		|{\mathbf M}_2(B({\mathcal F}))|>|B({\mathcal F})|>|{\mathbf M}_1(B({\mathcal F}))|.
	\end{equation}
	Then we shrink this domain so that in the chart $z_{\mathcal F}$ we have
	\begin{equation}
		\label{eqn:M1-prime-neq}
		\left|{\mathbf M}_1'(B({\mathcal F}))\right|\neq \left|\frac{{\mathbf M}_1(B({\mathcal F}))}{B({\mathcal F})}\right|.
	\end{equation}
	This is possible since ${\mathbf M}_1$ does not commute with ${\mathbf M}_3$, see \hyperref[proof:lem:LCs-dense]{proof} of \autoref{lem:LCs-dense}.
	Finally, we apply \autoref{lem:LCs-dense} choosing $k$, $l$, $m$ so large that
	\begin{equation}
		\label{eqn:k-l-large}
		|{\mathbf M}_1^{k+l}({\mathbf M}_2(B({\mathcal F})))|<|B({\mathcal F})|.
	\end{equation}

	Let $\gamma ^{(1)}$ be the representative of the class $[\gamma _2\gamma _1^{k+l}\gamma _2^m\gamma _1^{-k}]\in \pi _1(L_\infty )$ shown in \autoref{figure.c1}.
	\todo[inline,author=YK]{
		Should we describe $\gamma ^{(1)}$ and $\gamma ^{(2)}$ for visually impared readers?
	}
	Let $c_1$ be the lifting of $\gamma ^{(1)}$ starting at $B({\mathcal F})$.
	Since $B({\mathcal F})$ is a~hyperbolic fixed point of the monodromy along $\gamma ^{(1)}$, $c_1$ is a~hyperbolic limit cycle.
	Clearly, $c_1$ survives under a~small perturbation.

	\begin{figure}
		\centering
		\includegraphics{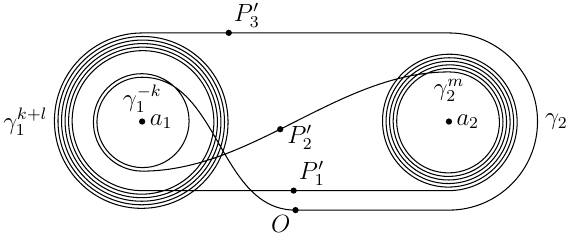}
		\caption{First cycle for $k=2$, $l=3$, $m=6$.}
		\label{figure.c1}
		\vspace{7mm}
		\includegraphics{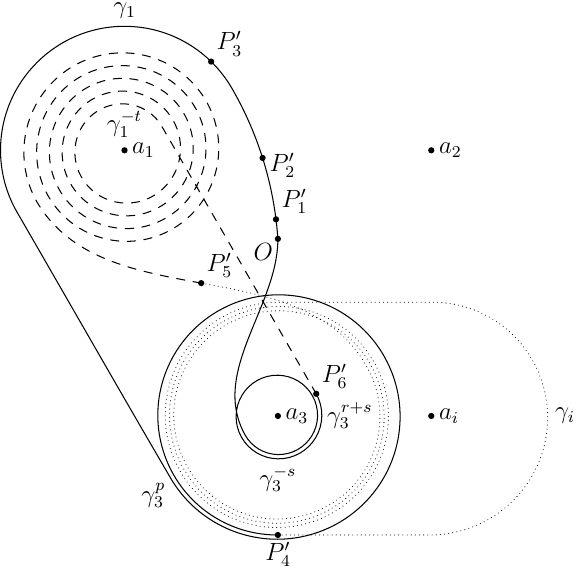}
		\caption{Second cycle for $p=2$, $q=1$, $r=1$, $s=2$, $t=5$}
		\label{figure.c2}
	\end{figure}

	\paragraph{Construction of the second cycle.}
	Fix a foliation ${\tilde {\mathcal F}}\in {\mathcal M}$, and apply \autoref{volk.details} to ${\mathcal M}$, ${\tilde {\mathcal F}}$, the points $A({\mathcal F})={\mathbf M}_1(B({\mathcal F}))$ and $B({\mathcal F})=B({\mathcal F})$, and the maps ${\mathbf M}_3$ and ${\mathbf M}_1$.
	Due to this theorem,  there exist $i\in \{1,\dots ,n+1\}$, $p\in {\mathbb N}$, $q\in \{0, 1\}$, $r\in {\mathbb N}$, $s\in {\mathbb N}$, $t\in {\mathbb N}$ such that the equality
	\[
		\left( {\mathbf M}_3^{-s}\circ {\mathbf M}_1^{-t}\circ {\mathbf M}_3^{r+s}\circ {\mathbf M}_i^q\circ {\mathbf M}_3^p\circ {\mathbf M}_1 \right)(B({\mathcal F}))=B({\mathcal F})
	\]
	defines a codimension one submanifold in some neighborhood ${\mathcal U}$, ${\tilde {\mathcal F}}\in {\mathcal U}\subset {\mathcal M}$.
	Let $\gamma ^{(2)}$ be the representative of $[\gamma _1\gamma _3^p\gamma _i^q\gamma _3^{r+s}\gamma _1^{-t}\gamma _3^{-s}]$ shown in \autoref{figure.c2}, let $c_2$ be the corresponding limit cycle.

	The cycles $c_1$ and $c_2$ are constructed.
	Let us prove that for $p$, $r$, $s$ large enough they satisfy the assertions of \autoref{one.handle}.

	\paragraph{Intersection of the cycles.}
	First, let us prove that for $p$ and $s$ large enough, $c_1\cap c_2=\set{B({\mathcal F})}$.
	Since $\gamma ^{(1)}$ and $\gamma ^{(2)}$  are the projections of $c_1$ and $c_2$ to the line at infinity, $c_1$ can intersect $c_2$ only above the intersection points of $\gamma ^{(1)}$ and $\gamma ^{(2)}$.
	Let $P_j$ be the points of $c_2$ that project to $P_j'$, $j=1,\dots ,6$ (see \autoref{figure.c2}).
	We shall deal separately with  arcs $P_6BP_4$, $P_4P_5$ and $P_5P_6$ of $c_2$.

	\begin{description}
		\item[Arc $P_6BP_4$ (solid line):]
			The arc $P_6'OP_4'$ of $\gamma ^{(2)}$ intersects $\gamma ^{(1)}$ in $\set{O, P_1', P_2', P_3'}$, hence we need to prove that none of $P_1$, $P_2$ and $P_3$ belongs to $c_1$.
			If $P_1\in c_1$, then $B$ is a fixed point of the monodromy map along the union of two arcs, $OP_1'$ on $\gamma ^{(1)}$ and $P_1'O$ on $\gamma ^{(2)}$.
			Thus  $({\mathbf M}_1^{k+l}\circ {\mathbf M}_2)(B)=B$, which contradicts \eqref{eqn:k-l-large}.
			Analogously, for~$P_2$ we get ${\mathbf M}_1^{-k}(B)=B$, but $B$ belongs to the domain of the linearizing chart of ${\mathbf M}_1$, so this is also impossible.
			For $P_3$ we get ${\mathbf M}_2(B)=B$, which contradicts \eqref{eqn:B-ineq-M1-M2}.
		\item[Arc $P_4P_5$ (dotted line):]
			Note that for large $p$ the arc $P_4P_5$ is much closer to $L_\infty $ than $c_1$, hence this arc is disjoint with $c_1$.
			Indeed, the $v$-coordinate of $P_4$ is $O(|\mu _3|^p)$ as $p\to \infty $.
			As we move along $\gamma _i$, the $v$-coordinate is multiplied by a bounded number.
			Then, as we make $r+s$ turns around $a_3$, we come even closer to the line at infinity.
			Therefore, $P_4P_5$ is $O(|\mu _3|^p)$-close to $L_\infty $.
			Recall that $p$ can be chosen arbitrarily large \emph{after} the choice of $c_1$, hence we can choose it so large that all points of $P_4P_5$ are much closer to $L_\infty $ than all points of $c_1$, in particular $P_4P_5$ does not intersect $c_1$.
		\item[Arc $P_5P_6$ (dashed line):]
			Similarly, going from $B$ along the arc $OP_6'P_5'\subset \gamma ^{(2)}$ in the opposite direction, we can see that all points of $P_5P_6$  are $O(|\mu _3|^s)$-close to $L_\infty $.
			Since $s$ can be chosen arbitrarily large \emph{after} we fix all other numbers, this part of $c_2$ can be made much closer to $L_\infty $ than all points of $c_1$, hence it does not intersect~$c_1$.
	\end{description}

	Therefore, $B({\mathcal F})$ is the only intersection point of $c_1$ and $c_2$.
	Note that this intersection is transverse, because the projections of $c_1$ and $c_2$ to $L_\infty $ intersect transversally.

	\paragraph{Hyperbolicity of the cycles.}
	Now let us prove that for sufficiently large numbers in \autoref{volk.details}, $c_2$ is a hyperbolic cycle.
	Indeed, due to \autoref{volk.details}, the derivative ${\mathbf M}_{\gamma ^{(2)}}'(B({\mathcal F}))$ can be made arbitrarily close to ${\tilde {\mathbf M}}_1'(B({\tilde {\mathcal F}}))\frac{B({\tilde {\mathcal F}})}{{\tilde {\mathbf M}}_1(B({\tilde {\mathcal F}}))}$.
	Here the former expression does not depend on the choice of coordinates, and the latter one is evaluated in $z_{\tilde {\mathcal F}}$.
	Due to \eqref{eqn:M1-prime-neq}, this ratio does not belong to the unit circle, hence one can choose $\alpha $ and $\beta $ in \autoref{volk.details} so that $c_2$ is a hyperbolic cycle.

	This completes the proof of the lemma, hence the proof of \autoref{positive.genus}.
\end{proof}

\section{Leaves of infinite genus}
\label{sec:inf-genus}

In this section we shall prove \autoref{infinite.genus}.

Consider the map $F_2:{\mathbb C}^2\to {\mathbb C}^2$ given by $(x, y) \mapsto  (z, w)=(x, y^2)$.
The image of ${\mathcal F}\in {\mathcal A}_n^{sym}$ under this map is a well-defined polynomial foliation.
More precisely, let $p$ and $q$ be polynomials such that $P(x, y) = yp(x, y^2)$ and $Q(x, y) = \frac 12 q(x, y^2)$.
Though the image of the \emph{vector field} \eqref{polynomial.vf} under $F_2$ is not well-defined, it is defined up to the sign change.
The additional time change $d\tau =y\,dt$ yields a~well-defined foliation ${\dot z}=p(z, w)$, ${\dot w}=q(z, w)$ which we denote by~$(F_2)_*{\mathcal F}$

It is easy to check that $F_2$ sends the leaves of ${\mathcal F}$ to the leaves of $(F_2)_*{\mathcal F}$ if we remove the points $\set{(x, 0)|q(x, 0)=0}$ both from domain and codomain.
Moreover, for each leaf $L$ of ${\mathcal F}$, the map $F_2|_L$ is a branched double covering with branch points at $L\cap \set{y=0}$.

The following lemma explicitly describes the open and dense subset of ${\mathcal A}_n^{sym}$ (or ${\mathcal B}_{n+1}^{sym}$) that satisfies the assertion of \autoref{infinite.genus}.

\begin{lemma}
	\label{infinite.genus.details}

	Consider a foliation ${\mathcal F}\in {\mathcal A}_n^{sym}$ such that

	\begin{itemize}
		\item ${\mathcal F}$ has no rational first integral;
		\item $(F_2)_*{\mathcal F}$ has no singular points on the projective line $\{w=0\}\subset {\mathbb C}P^2$.
	\end{itemize}

	Then ${\mathcal F}$ has finitely many (possibly zero) algebraic leaves, and all other leaves have infinite genus.
\end{lemma}

\begin{remark}
	We can also take the saturation of the set constructed above by the orbits of affine group.
	This adds $3$ to the dimension, but this saturation will be a more complicated object than an open dense subset of a linear subspace.
\end{remark}

\begin{proof}
	Since ${\mathcal F}$ has no rational first integral, \autoref{few-algebraic-leaves} implies that all but a finite number of leaves are non-algebraic.
	Let $L$ be a non-algebraic leaf of ${\mathcal F}$.
	Due to \autoref{infinite-intersection}, $F_2(L)$ intersects $\set{w=0}$ in infinitely many points.
	Since $q$ vanishes in at most finite number of these points, $L$ intersects $\set{y=0}$ in infinitely many points as well.

	Recall that the restriction $F_2|_L:L\to F_2(L)$ is a ramified double covering with branch points at $L\cap \set{y=0}$.
	Hence the covering $F_2|_L$ has countably many ramification points.
	Consider a closed disk $D\subset F_2(L)$ that contains $2N+1$ ramification points, none of them in $\partial D$.
	Then $\partial D$ lifts to a~circle, hence $F_2^{-1}(D)$ is a~disc with $g$ handles, $g\in {\mathbb Z}_+$.
	Due to \wiki{Riemann-Hurwitz formula}{Riemann--Hurwitz formula}\footnote{For historical reasons, Riemann--Hurwitz Formula deals only with branched coverings of \emph{closed} surfaces, though the proof is literally the same for surfaces with borders.},

	\[
		1-2g=\chi (F_2^{-1}(D))=2\chi (D)-2N-1=1-2N.
	\]

	Thus $g=N$.
	Finally, $L$ has infinite genus.
\end{proof}

\begin{remark}
	The last step can be done in a more intuitive manner.
	Consider infinitely many pairwise disjoint discs $D_i$, each contains three ramification points.
	For each disc consider two loops $c_1$, $c_2$ shown in \autoref{figure.cycles-handle}.
	Since there are exactly two ramification points inside each cycle $c_1$, $c_2$, the monodromy along $c_i$ does not swap the branches of $F_2$.
	Let ${\tilde c}_1$ and ${\tilde c}_2$ be the lifts of $c_1$, $c_2$ starting at the same preimage of $O$.
	The cycle going from $O$ to $P$ along $c_1$, then back to $O$ along $c_2$ has only one ramification point inside, hence ${\tilde c}_1$ does not meed ${\tilde c}_2$ in $F_2^{-1}(P)$.
	Therefore, ${\tilde c}_1$ and ${\tilde c}_2$ generate a~handle.
\end{remark}

\begin{figure}
	\centering
	\includegraphics{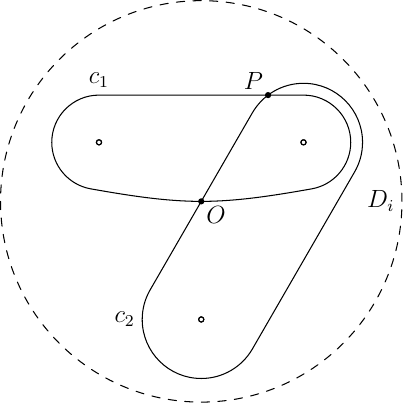}
	\caption{Two cycles whose preimages under $F_2$ generate a~handle.}
	\label{figure.cycles-handle}
\end{figure}

Since ${\mathcal B}_n\subset {\mathcal A}_n$, this lemma is applicable to foliations ${\mathcal F}\in {\mathcal B}_n^{sym}$ as well.

Now let us deduce \autoref{infinite.genus} from the above lemma.

\begin{proof}
	[Proof of \autoref{infinite.genus}]
	It is sufficient to prove that for $n\geq 2$ the subset of ${\mathcal A}_n^{sym}$ (resp., ${\mathcal B}_{n+1}^{sym}$) defined by the additional assumptions from \autoref{infinite.genus.details} is open and dense in the ambient projective space.

	Let us prove that a~generic foliation ${\mathcal F}\in {\mathcal A}_n^{sym}$ or ${\mathcal F}\in {\mathcal B}_{n+1}^{sym}$ has no rational first integral.
	Note that a complex hyperbolic singular point is not \emph{locally} integrable, hence a foliation with a~complex hyperbolic singular point cannot have a rational first integral.
	Since a complex hyperbolic singular point survives under small perturbations, it is sufficient to prove that the set of foliations from ${\mathcal A}_n^{sym}$ (resp., ${\mathcal B}_{n+1}^{sym}$) having a complex hyperbolic singular point is dense in the ambient space.

	Consider a foliation ${\mathcal F}_0$ from ${\mathcal A}_n^{sym}$ or ${\mathcal B}_{n+1}^{sym}$, $n\geq 2$.
	Let $(x_0, y_0)$ be one of its singular points with $y_0\neq 0$.
	Let $A_0=\begin{pmatrix} a & b \\ c & d\end{pmatrix}$ be its linearization matrix at $(x_0,y_0)$.
	Consider the two-parametric perturbation ${\mathcal F}_{\eps ,\delta }$ of ${\mathcal F}_0$, $\eps ,\delta \in ({\mathbb C}, 0)$, given by

	\begin{align*}
		{\dot x} &= P_0(x, y) + y(x-x_0) \eps ;\\
		{\dot y} &= Q_0(x, y) + (y^2 - y_0^2) \delta ,
	\end{align*}

	It is easy to see that the perturbed foliation belongs to the same class (${\mathcal A}_n^{sym}$ or ${\mathcal B}_{n+1}^{sym}$) and has a singularity at the same point $(x_0,y_0)$.
	The linearization matrix of ${\mathcal F}_{\eps ,\delta }$ at $(x_0, y_0)$ is $A_{\eps ,\delta }=\begin{pmatrix}a + y_0\eps & b \\ c & d+2y_0\delta \end{pmatrix}$.
	Clearly,

	\begin{align*}
		\tr A_{\eps ,\delta }-\tr A_0&=y_0(\eps +2\delta );\\
		\det A_{\eps ,\delta }-\det A_0&=y_0(\eps d+2\delta a+2y_0\eps \delta ).
	\end{align*}

	It is easy to see that we can achieve any small perturbation of the trace and determinant of the linearization matrix.
	Therefore, we can achieve any small perturbation of the eigenvalues.
	In particular, after some perturbation the singular point at $(x_0, y_0)$ becomes complex hyperbolic.

	The line $w=0$ contains no singular points of $(F_2)_*{\mathcal F}$ for a generic ${\mathcal F}$ since

	\begin{itemize}
		\item $(p|_{\set{w=0}}, q|_{\set{w=0}})$ may be any pair of polynomials of degrees $\deg {\mathcal F}-1$ and $\deg {\mathcal F}$, respectively, hence for a generic ${\mathcal F}$ they have no common roots;
		\item $(F_2)_*{\mathcal F}$ has a singularity at the intersection point of $\set{w=0}$ and the line at infinity if and only if $\deg q|_{\set{w=0}}<\deg {\mathcal F}$, and this is false for a generic ${\mathcal F}$.
	\end{itemize}
\end{proof}

\begin{remark}
	The arguments above do not work for ${\mathcal A}_1^{sym}$ and ${\mathcal B}_2^{sym}$ because generic foliations from these spaces have rational first integrals.
	Indeed, a generic foliation from the former space is affine equivalent to a foliation of the form

	\begin{align*}
		{\dot x}&=y;\\
		{\dot y}&=ax,
	\end{align*}

	which has the first integral $y^2-ax^2$.
	A generic foliation from ${\mathcal B}_2^{sym}$ is affine equivalent to a foliation of the form

	\begin{align*}
		{\dot x} &= xy;\\
		{\dot y} &= x+y^2+a,
	\end{align*}

	which has the first integral $\frac{(x +a)^2 - a y^2}{x^2}$.
\end{remark}

\appendix

\section{Proof of Ilyashenko's Theorem}
\label{sec:proof-Il}

In this Section we shall prove the following theorem.

\begin{theorem}
	For $n\geq 2$, let ${\mathcal A}_n^h\subset {\mathcal A}_n$ be the space of foliations given by homogeneous polynomials $P$ and $Q$.
	For a foliation ${\mathcal F}$ from some open dense subset of ${\mathcal A}_n^h$, all its leaves except for a finite set have infinite genus.
\end{theorem}

This theorem was proved by Yu.~Ilyashenko many years ago, and he communicated this proof to various people orally.
As far as we know, the proof was never published by him or anyone else.

\begin{proof}
	Take a homogeneous foliation ${\mathcal F}$.
	Note that the polynomials ${\tilde P}$ and ${\tilde Q}$ in \eqref{vf-uv} do not depend on $u$, hence in the chart $(u,v) = (\frac 1x, \frac yx)$ our foliation ${\mathcal F}$ is given by

	\begin{align*}
		{\dot u} &= u{\tilde P}(v)\\
		{\dot v} &= h(v)
	\end{align*}

	Clearly, the monodromy group at infinity is generated by linear maps ${\mathbf M}_j:u\mapsto \mu _j u$.

	Fix a cross-section $S$ given by $v=\const$ and a point $p\in S\setminus L_\infty $.
	Let us find a handle passing through $p$.

	The monodromy maps along loops $[\gamma _2,\gamma _1]=\gamma _2\gamma _1\gamma _2^{-1}\gamma _1^{-1}$ and $[\gamma _3,\gamma _2^{-1}]=\gamma _3\gamma _2^{-1}\gamma _3^{-1}\gamma _2$ are identity maps, hence the lifts $c_1$ and $c_2$ of these loops starting at $p$ are closed loops.

	\begin{figure}
		\centering
		\includegraphics{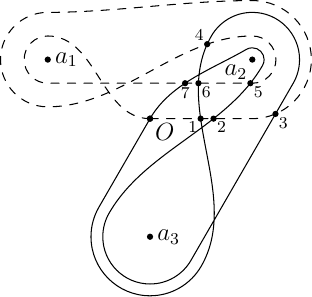}
		\caption{Cycles $c_1$ (dashed) and $c_2$}
	\end{figure}

	Note that these loops intersect only at $p$.
	Indeed, if $c_1$ and $c_2$ intersect above one of $7$ other intersection points of $[\gamma _2,\gamma _1]$ and $[\gamma _3,\gamma _2^{-1}]$, then $p$ is a fixed point of one of the maps ${\mathbf M}_3$, ${\mathbf M}_2$, ${\mathbf M}_2^{-1}\circ {\mathbf M}_3$, ${\mathbf M}_2^{-1}\circ {\mathbf M}_1^{-1}\circ {\mathbf M}_3$, ${\mathbf M}_1\circ {\mathbf M}_2$, ${\mathbf M}_1^{-1}\circ {\mathbf M}_3$, ${\mathbf M}_1$, respectively.
	In a generic case, all these maps are linear non-identical, thus they have no fixed points except zero.

	Thus $c_1$ and $c_2$ intersect transversally at one point, so we have found a handle passing through~$p$.

	Consider a leaf $L$ of ${\mathcal F}$ which is not a separatrix $v=a_j$ of a singular point of ${\mathcal F}$ at the line at infinity.
	In a generic case (say, if $|\mu _j|\neq 1$ for some $j$), $L$ intersects $S$ arbitrarily close to $L_\infty $.
	Since each intersection point produces a handle, $L$ has infinite genus.
	Hence any leaf except separatrices at $a_j$ has infinite genus.
\end{proof}

\section*{Acknowledgements}

We proved these results and wrote the first version of this article during our 5-months visit to Mexico (Mexico City, then Cuernavaca) in 2014.
We are very grateful to UNAM (Mexico) and HSE (Moscow) for supporting this visit.
Our deep thanks to Laura Ortiz Bobadilla, Ernesto Rosales-González and Alberto Verjovsky, for invitation to Mexico and for fruitful discussions.
It was Laura Ortiz who pointed us to the question about the leaves with infinite genus.

We are thankful to Tatiana Firsova who gave us an idea of a simplification in the second part of the article.
We are also grateful to Yulij Ilyashenko for permanent encouragement, and to Victor Kleptsyn for useful discussions.

Our big gratitude to Valente Ramirez who read first versions of this article and made a lot of useful suggestions.
We are very grateful to Bertrand Deroin whose corrections and comments motivated us to make the article more rigorous, and much more readable.
\printbibliography
\newpage
\listoftodos
\end{document}